\documentclass[12pt]{amsart}
\usepackage{amsmath,amssymb,amsfonts,amsthm, amscd,indentfirst}
\usepackage{hyperref}
\usepackage{enumitem}
\usepackage{color}
\numberwithin{equation}{section}
\def\di{\displaystyle}
\def\O{\Omega}
\def\e{\epsilon}
\def\S{\Sigma} 
\def\n{\nabla}
\def\P{\partial}
\def\p{\partial}

\def\bn{\overline \nabla}
\def\bd{\overline \Delta}
\def\a{\alpha}
\def\b{\beta}
\def\n{\nabla}

\def\tr{\mathrm{tr}}

\def\O{\Omega}
\def\p{\partial}
\def\e{\epsilon}
\def\a{\alpha}
\def\b{\beta}
\def\g{\gamma}

\def\k{\kappa}
\def\l{\lambda}

\def\De{\Delta}
\def\n{\nabla}
\def\<{\langle}
\def\>{\rangle}
\def\div{{\rm div}}
\def\De{\Delta}
\def\n{\nabla}
\def\ode{\overline{\Delta}}
\def\on{\overline{\nabla}}

\def\SS{\mathbb{S}}
\def\HH{\mathbb{H}}
\def\tr{\mathrm{tr}}

\def\O{\Omega}
\def\p{\partial}
\def\e{\epsilon}
\def\a{\alpha}
\def\b{\beta}
\def\g{\gamma}

\def\l{\lambda}

\def\wh{\widehat}

    \newtheorem{prop}{Proposition}[section]
    \newtheorem{theo}[prop]{Theorem}
    \newtheorem{lemm}[prop]{Lemma}
    \newtheorem{coro}[prop]{Corollary}
    \newtheorem{rema}[prop]{Remark}
    
    \newtheorem{defi}[prop]{Definition}

\begin{document}
\title[An integral formula on sub-static manifolds]{An integral formula  and its applications on
    sub-static manifolds
}

\author{Junfang Li and Chao Xia}

\address{Department of Mathematics\\
University of Alabama at Birmingham\\
Birmingham, AL 35294}
\email{jfli@uab.edu}

\address{School of Mathematical Sciences\\
Xiamen University\\
361005, Xiamen, P.R. China}
\email{chaoxia@xmu.edu.cn}
\thanks{Research of CX is  supported in part by the Fundamental Research Funds for the Central Universities (Grant No. 20720150012), NSFC (Grant No. 11501480)  and CRC Postdoc Fellowship. }

\begin{abstract}
    In this article, we first establish the main tool - an integral formula (\ref{equ10.5 nc}) for Riemannian manifolds with multiple boundary components (or without boundary). This formula  generalizes Reilly's original formula from \cite{Re2} and the recent result from \cite{QX}. It provides a robust tool for sub-static manifolds regardless of the underlying topology.  
    
    Using (\ref{equ10.5 nc}) and suitable elliptic PDEs, we prove Heintze-Karcher type  inequalities for bounded domains  in general sub-static manifolds which recovers some of the results from Brendle \cite{Br} as special cases.
    
    On the other hand, we prove a Minkowski inequality for {\it static convex} hypersurfaces in a sub-static warped product manifold. Moreover, we obtain an almost Schur lemma for {\it horo-convex} hypersurfaces in the hyperbolic space and {\it convex} hypersurfaces in the hemi-sphere, which can be viewed as a special Alexandrov-Fenchel inequality.
\end{abstract}

\date{}
\maketitle

\section{Introduction}
Reilly established a series of integral formulas on hypersurfaces in Euclidean space in  \cite{Re1} in early 70s and later generalized them to Riemannian manifolds in \cite{Re2}. Since then, this type of integral formulas have been used for many problems in geometric analysis during the past decades. Reilly's formula is particularly useful for manifolds with nonnegative Ricci curvature. However, for rest of the cases, not much information can be obtained from the formula. Very recently, the second named author together with Qiu, made a first try to establish a formula which is useful when manifolds allow negative curvature \cite{QX}.
In particular, they can prove geometric inequalities which are sharp in space forms.

Along the line of \cite{QX}, we will prove a more general integral formula which includes the previous work of \cite{Re1, Re2, QX} as special cases. More interestingly, this new integral formula leads to geometric inequalities which are sharp in a much larger class of Riemannian manifolds besides the space forms. 

The rest of the introduction will split into several parts. First, we introduce the new integral formula. Then we give several Heintze-Karcher type inequalities which can be proved by the new formula. We note these inequalities, when on substatic warped product spaces, are due to Brendle \cite{Br} by a completely different method. In the end, we will show that our elliptic method can be used to prove Minkowski type inequalities and almost Schur type theorems.

The most general form of our integral formula is presented in Section \ref{sec 3}, see Theorem \ref{theo grf}. For applications in this article, we will only focus on the following special case.

\begin{theo}
    \label{theo rf nc}
     Let $(M^n, \bar g)$ be an $n$-dimensional  smooth Riemannian manifold and $\O\subset M$ be a bounded domain with smooth boundary $\p \O$. Let $V\in C^{\infty}(\overline{\O})$ be a smooth function 
such that $\frac{\bn^2 V}{V}$ is continuous up to $\p\O$. 
     
Then for any $f\in C^\infty(\overline{\O})$, 
the following integral identity holds: 
\begin{eqnarray}        \label{equ10.5 nc}
           && \di\int_{\O}V \left(\bd f-\frac{\bd V}{V}f\right)^2-V\left|\bn^2 f-\frac{\bn^2 V}{V}f\right|^2 d\O\\&=&\di\int_{\p\O}V h(\n z,\n z)+2V u\Delta z+V Hu^2+V_{,\nu}|\n z|^2+2z\bn^2 V(\n z, \nu) dA\nonumber\\ 
            &&+\di\int_{\p\O} -2zu (\De V+HV_{,\nu})-z^2\frac{\bn^2 V-\bd V \bar g}{V}(\bn V, \nu) dA\nonumber\\
            &&\di+  \int_{\O}V\cdot Q\left(\bn f-\frac{\bn V}{V}f, \bn f-\frac{\bn V}{V}f\right) d\O.\nonumber
\end{eqnarray}
Here $\nu$ is the outward unit normal vector, $z= f|_{\P\O}$, $ u=\bn_\nu f$, $V_{,\nu}=\bn_\nu V$, $h(\cdot ,\cdot)$ and $H$ are the second fundamental form and the mean curvature (the sum of principal curvatures) of $\p\O$ respectively. 
$Q$ is a $(0,2)$ tensor field on $\O$  defined by
\begin{align}
    V\cdot Q:= \bd V \bar g-\bn^2 V+V {\rm Ric},
    \label{def Q}
\end{align}
where ${\rm Ric}$ is the Ricci curvature of $(M,\bar g)$.
\end{theo}

\ 

For more information on the notations, we refer to the beginning of Section \ref{sec 2}.  
Clearly, when $V\equiv 1$ and $V\cdot Q={\rm Ric}$, we recover Reilly's original formula. Just as in Reilly's original formula where non-negative Ricci in the interior integral is crucial for any likely applications, non-negative definite $(0,2)$-tensor $Q$ in our integral formula plays a crucial role in our setting.  

For convenience, we introduce a Riemannian triple $(M^n, \bar g, V)$ which constitutes an $n$-dimensional  smooth connected manifold $M$, a Riemannian metric $\bar g$ on $M$ and  a smooth nontrivial 
function $V$ on $M$. Throughout this paper, we assume $n\geq 3$. We give the following definitions. 
\begin{defi}\label{defi static}
A Riemannian triple $(M, \bar g, V)$ is called {\em static} if 
\begin{eqnarray}\label{static potential}
\bd V\bar g-\bn^2 V+V{\rm Ric}= 0.
\end{eqnarray}
A Riemannian triple $(M, \bar g, V)$ is called {\em sub-static} if there exists a smooth $(0,2)$-tensor $Q$ on $M$ such that
\begin{eqnarray}\label{static potential inequality}
    VQ:=\bd V\bar g-\bn^2 V+V{\rm Ric}\ge 0.
\end{eqnarray}
In both cases, we call $V$ a {\em potential function}.
\end{defi}

Our definition of {\rm static} and {\rm sub-static} Riemannian triples comes naturally from and has roots in the study of static spacetime in general relativity. More background details can be found in Section \ref{sec 2}.  

\

We now focus on using the integral formula to prove geometric inequalities for domains in static or sub-static Riemannian triples. Some of the inequalities was initiated by Brendle et. al. \cite{Br, BHW}.

As the first application of (\ref{equ10.5 nc}), we establish the Heintze-Karcher type inequalities. Similar type inequalities have been studied \cite{Br} using a parabolic approach, see also  \cite{WWZ}.

Different from the previous work, besides assuming the Riemannian triple is sub-static, we do not impose any rotational symmetry or topological constraints for the ambient manifolds. In particular, we allow the domain to have multiple boundary components which corresponds to a bounded region containing multiple event horizons on a time slice of a static spacetime. 
 
\begin{theo}[Heintze-Karcher type inequality]
    \label{HK theo}
    Let $(M^n, \bar g, V)$ be an $n$-dimensional sub-static Riemannian triple. Let $\O\subset M$ be a bounded domain with $\p \O= \S\cup \bigcup_{l=1}^{\tau} N_l$ satisfying the {\em inner boundary condition} defined in Definition \ref{homo} and suppose the outermost boundary hypersurface $\S$ is {\em mean convex}. 
 Assume $V>0$ in $\O$. 
 
Then the following inequality holds: 
     \begin{equation}
         n\di\left(\int_{\O}Vd\O+\sum^\tau_{l=1}c_l\int_{N_l}  V_{,{\nu}} dA\right)\le\di (n-1)\int_{\S}\frac{V}{H} dA,
         \label{Heintze-Karcher}
     \end{equation}
     where $$c_l =\di\max_{z\in N_l}\left\{- \di\frac{n-1}{n}\frac{V(z)}{(\bd V \bar g_{\nu\nu}-\bn^2_{\nu\nu} V)(z)}\right\}<0$$ are constants.
Moreover, if equality in \eqref{Heintze-Karcher} holds, then $\S$ is umbilical.
     
     When $\tau=0$, i.e. $\P\O$ has only one component $\S$, then \eqref{Heintze-Karcher} is reduced to
\begin{equation}
         n\di\int_{\O}V d\O\le\di (n-1)\int_{\S}\frac{V}{H} dA.\\
         \label{Heintze-Karcher null}
     \end{equation}
\end{theo}

For relation between the constant $c_l$,  the scalar curvature ${\rm R}$ of $(M, \bar g)$ and the scalar curvature ${\rm R}_{N_l}$ of $N_l$, see Remark \ref{rema homo} (4). By assuming a {\em locally warped product structure} around the horizon(s), c.f. Definition \ref{local warped}, we have the following results similar to the work in \cite{Br} but with possibly finitely many horizons. 

    
    \begin{coro}
        \label{coro local}
        Let $(M^n, \bar g, V)$ and $\O$ be as Theorem \ref{HK theo}. We further assume $M^n$ has local warped product structure  $\bar g=dr^2+\lambda_l^2(r)g_{N_l}, r\in [0,\e_l)$ for a small $\e_l>0$ around its inner boundary component $N_l$. 
 Assume near the boundary $N_l$, $V=\l'_l(r)$ and $\l''_l(0)\neq 0$. 
 
Then the following inequality holds:
     \begin{equation}
         \di n\int_{\O}V d\O+\sum^\tau_{l=1}\lambda_l(0)^n{\rm Vol}(N_l, g_{N_l}) \le\di (n-1)\int_{\S}\frac{V}{H} dA,
         \label{Heintze-Karcher local}
     \end{equation}
     where ${\rm Vol}(N_l, g_{N_l})$ is the volume of the boundary component with respect to its metric $g_{N_l}$. Moreover, if equality in \eqref{Heintze-Karcher local} holds, then  the boundary $\S$ is umbilical.
    \end{coro}
\

As in the works of Ros \cite{Ros} and Brendle \cite{Br}, the Heintze-Karcher inequality together with a Minkowski identity (or inequality) immediately imply an Alexandrov type rigidity theorem. The following theorem is due to Brendle \cite{Br}. The only difference is when the base manifold $N$ satisfies ${\rm Ric}_N=(n-2)\rho g_N$, we drop conditions (H4) and (H4') to show umbilic hypersurfaces must be slice. 

\begin{theo}
    \label{Alex theo}
    {\bf (Brendle \cite{Br})} Let $M^n=[0,\bar r)\times N^{n-1}$ be a  sub-static warped product manifold with metric $\bar g=dr^2+\l(r)^2g_N$, where $N$ is a closed manifold with ${\rm Ric}_N=(n-2)\rho g_N$ for some $\rho\in \mathbb{R}$, and the potential function $V=\l'(r)$ satisfies \ref{cond 1} or \ref{cond 2} (see Section \ref{sec 2.2}). 
 
Let $\S\subset M$ be a closed orientable embedded hypersurface  with constant mean curvature. Then $\S$ is either a geodesic slice of an event horizon or it is a geodesic sphere lying  in the region with constant sectional curvature.
\end{theo}

Several remarks about Theorem \ref{HK theo}--\ref{Alex theo} are given below.
\begin{rema}
\
\begin{itemize}  
\item[(i)] Theorem \ref{HK theo} works for domains in general sub-static Riemannian manifolds and allows more than two boundary components. 
    
\item[(ii)] Our proof is completely different from \cite{Br} and avoids the subtle lemmata and Greene-Wu's approximation technique.

\item[(iii)] 

In Theorem \ref{Alex theo}, by adding ${\rm Ric}_N=(n-2)\rho g_N$, we remove conditions $(H4)$ and $(H4')$ from \cite{Br}. 


\end{itemize} 
\end{rema}

\

In the past several decades, we have witnessed much successes of applying geometric flows to prove geometric inequalities and there have been numerous references which we will not list here. As an immediately related example, Brendle's proof on the Heintz-Karcher type inequality \cite{Br} is based on a parabolic approach. The key of a successful parabolic approach is finding monotonic quantities along the flow. When such a property of monotonicity  is not available, we may have to seek other ways. For instance, the following classical Minkowski inequality for convex domains in $\mathbb R^n$, 
\begin{align} 
\left(\int_\S dA\right)^2\geq \frac{n}{n-1}\int_\Omega d\O \int_\S H dA,
    \label{Mink Euclidean}
\end{align}
describes a sharp relation among three geometric quantities, the volume of a bounded convex domain, the surface area of the smooth boundary, and the total mean curvature over the boundary. To our knowledge, there is no geometric flow approach for this type of inequalities. It seems unlikely to find a monotonic quantity involving three quantities. 

By an elliptic method together with \eqref{equ10.5 nc}, we can generalize Minkowski type inequality (\ref{Mink Euclidean}) to sub-static Riemannian triples. For simplicity, we only state the results for domains with one or two boundary components. Domains with more boundary components can be similarly considered.
 
\begin{theo}\label{rm22}     Let $(M^n, \bar g, V)$ be an $n$-dimensional sub-static Riemannian triple. Let $\O\subset M$ be a bounded domain  with a connected boundary $\p\O=\S$ on which $V>0$.
We further assume the second fundamental form of $\S$ satisfies \begin{eqnarray}\label{ass}
h_{\a\b}- \frac{V_{,\nu}}{V} g_{\a\b}\geq 0.
\end{eqnarray}

Then the following inequality holds:
\begin{eqnarray}\label{rm2}
\left(\int_\S VdA\right)^2\geq \frac{n}{n-1}\int_\Omega Vd\O \int_\S HV dA.
\end{eqnarray}
Moreover, if there is one point in $\S$ at which the strict inequality holds in \eqref{ass}, and equality in \eqref{rm2} holds, then $\S$ is umbilical and of constant mean curvature.
\end{theo}

\begin{theo}\label{rm33} Let $M^n=[0,\bar r)\times N^{n-1}$ be a sub-static warped product manifold with metric $\bar g=dr^2+\l(r)^2g_N$, where $N$ is an $(n-1)$-dimensional closed manifold and the potential function $V=\l'(r)$ satisfies  \ref{cond 2} (see Section \ref{sec 2.2}).
 Assume $\S$ is homologous to $\p M$, namely, there is a bounded domain $\O$ such that $\p\O=\S\cup \p M$. We further assume the second fundamental form of $\S$ satisfies \eqref{ass}. 

Then the following inequality holds:
     \begin{eqnarray}\label{case2Minkowski}
        && \left(\int_\S VdA\right)^2\geq \frac{n}{n-1}\left(\int_\Omega Vd\O+\frac1n\l(0)^n{\rm Vol}(N, g_N)\right) \int_\S HV dA.
     \end{eqnarray}
 Moreover,  if there is one point of $\S$ at which the strict inequality holds in \eqref{ass}, then equality in \eqref{case2Minkowski} holds if and only if  $\S$ is a slice $N\times \{r\}$ for some $r\in(0, \bar r)$.
\end{theo}

Theorem \ref{rm22} recovers the result of the second named author in \cite{Xia} where the ambient spaces $M$ were restricted to space forms. We follows closely but simplifies the argument in \cite{Xia}. 
As observed in \cite{Xia}, the condition \eqref{ass} is satisfied when $\S$ is horoconvex in the hyperbolic space or $\S$ is convex in the hemi-sphere. 
This strong convexity is named as {\it static convex} by Brendle-Wang  \cite{BW} for its correspondence in static spacetime. 

\

In the end, we prove  a weighted Alexandrov-Fenchel inequality or a weighted almost Schur lemma for closed hypersurfaces in space forms.
We use $\HH^n$ to denote the hyperbolic space with curvature $-1$ and $\SS_+^n$ to denote the open hemi-sphere with curvature $1$.
\begin{theo}\label{almost Schur}
Let $\S$ be a smooth, closed hypersurface in  $\mathbb{H}^n$ ($\SS_+^n$ resp.). Let $V=\cosh r$ ($\cos r$ resp.), where $r(x)=dist(x,p)$ for some fixed point $p\in\HH^n$ ($p\in{\mathbb{S}_+^n}$ resp.).  In case of $\HH^n$ we assume that the second fundamental form of $\S$ satisfies \eqref{ass}, in particular, assume $\S$ is horo-convex. In case of $\SS^n_+$ we assume $\S$ is convex.

Then the following inequality holds: \begin{eqnarray}\label{AF1}
&&\int_\S V\left|H-\overline{H}^V\right|^2 dA\leq \frac{n-1}{n-2} \int_\S V\left|h_{\a\b}-\frac{H}{n-1}g_{\a\b} \right|^2 dA.
\end{eqnarray}
where $\overline{H}^V=\frac{\int_\S HV dA}{\int_\S V dA}$. Equivalently, 
\begin{eqnarray}\label{AF2}
&&\left(\int_\S HV dA\right)^2\geq \frac{2(n-1)}{n-2}\int_\S V dA \int_\S \sigma_2 (h) V dA,
\end{eqnarray}
where $\sigma_2(h)=\frac{1}{2}(H^2-|h|^2)$ is the second mean curvature of $\S$.
Equality holds in \eqref{AF1} and \eqref{AF2} if and only if $\O$ is a geodesic ball $B_R(q)$ for some point $q\in \HH^n$ ($q\in{\mathbb{S}_+^n}$ resp.).
\end{theo}

When $V\equiv 1$, inequality \eqref{AF2} is a special case of the classical Alexandrov-Fenchel inequality for convex hypersurfaces in the Euclidean space.
In view of \eqref{rm2} and \eqref{AF2}, we complete the picture of  the weighted Alexandrov-Fenchel inequality in $\HH^3$ and $\SS_+^3$. We believe the weighted Alexandrov-Fenchel inequality holds true in general.

The almost Schur lemma in terms of  scalar curvature has been shown by de Lellis-Topping \cite{dLT}. Later Ge-Wang \cite{GW} observed the equivalence of the almost Schur lemma and the Alexandrov-Fenchel type inequality.  For closed hypersurfaces in space forms, an unweighted almost Schur lemma in terms of mean curvature has been proved by Perez \cite{Pe} along the line of \cite{dLT}, see also \cite{CZ}. Our proof of Theorem \ref{almost Schur} relies on the integral formula for closed hypersurfaces and the crucial observation that a horo-convex (convex resp.) hypersurface in $\HH^n$  ($\SS_+^n$ resp.) is a closed sub-static manifold. 

A similar weighted almost Schur lemma for sub-static closed manifolds, in the flavor of de Lellis-Topping, is also presented at the end of this paper.

\

The rest of the paper is organized as follows. In section 2, we introduce our notations and preliminaries. In Section 3, we prove the integral formula and also its variations. Section 4-6 will be devoted to prove the Heintze-Karcher type  inequality, the Minkowski type inequality and the Alexandrov-Fenchel type inequality. \\

We note that a similar result on Heintze-Karcher type inequality appeared on arXiv in \cite{WW} recently. However, the method in \cite{WW} is different. The authors were interested in static manifolds motivated from applications from general relativity. 

\section{Notations and preliminaries}\label{sec 2}

Throughout this paper, we denote by $\bn$, $\bd$ and $\bn^2$ the gradient, the Laplacian and the Hessian on $(M,\bar g)$ respectively, while by $g$, $\n$, $\De$ and $\n^2$ the induced metric, the gradient, the Laplacian and the Hessian on $\p\O$ respectively.  Let $\nu$ be the unit outward normal of $\p\O$. We denote by $h(X,Y)=\bar g(\on_X \nu, Y)$ and $H=\tr_{\bar g} h$ the second fundamental form and the mean curvature (with respect to $-\nu$) of $\p\O$ respectively. Let $d\O$ and $dA$ be the canonical measure of $(M,\bar g)$ and $\p\O$ respectively, which will be omitted frequently for simplicity. Let ${\rm Ric}$ and ${\rm R}$ be  the Ricci curvature tensor and the scalar curvature of $(M,\bar g)$ respectively. For simplicity,  we use a subscript to  denote the covariant derivative with respect to $g$, and a comma plus a subscript to denote that with respect to $\bar g$, for example ${f}_\a=\n_\a f$, ${\rm R}_{,i}=\bn_i {\rm R}$. We will use the Einstein summation convention and also the convention that the Latin indices run through $1$ to $n$ while the Greek indices run through $1$ to $n-1$.

Our definitions of the {\rm static} and the {\rm sub-static} Riemannian triples  in Definition \ref{defi static} has a deep background in the study of general relativity. Let $(M^n, \bar g)$ be a time slice of a {\em static spacetime} $\widehat{M}^{n+1}=\mathbb{R}\times M^n$ with Lorentzian metric $\widehat{g}= -V^2 dt^2+\bar g$.  
Then the $(0,2)$-tensor $Q$ has a natural physical meaning in the static spacetime. $(\widehat{M}, \widehat{g})$ is said to satisfy the {\em Null Convergence Condition (NCC)} if $\widehat{{\rm Ric}}(X, X)\geq 0$ for all null vector fields $X$ in $T\widehat{M}$. It is known that the sub-static condition of a Riemannian triple $(M, \bar g, V)$ is equivalent to the NCC condition of the corresponding static spacetime $(\widehat{M}, \widehat{g})$,  c.f. \cite{WWZ}. More specifically, $Q_{ij}=\wh R_{ij}X_iX_j$ under orthnormal frames.

It is also well-known that a static Riemannian triple must be of constant scalar curvature and the zero set $\{V=0\}$ is a totally geodesic regular hypersurface, see e.g. \cite{FM, Cor}.

\subsection{Local and global warped product spaces}\label{sec 2.2}
The class of warped product manifolds always serves as nontrivial examples in the theory of general relativity. We briefly review some elementary fact for warped product manifolds here.

Let $M^n=[0,\bar r)\times N^{n-1}$ be a  warped product manifold with metric $\bar g=dr^2+\l(r)^2g_N$, where $N$ is an $(n-1)$-dimensional closed manifold and $\l: [0,\bar r)\to \mathbb{R}$ is a smooth positive function. Usually the following two type of conditions are assumed:

    \begin{enumerate}[label=\textbf{H.\arabic*}]
\item \label{cond 1} $\l'(0)=1, \l''(0)=0$ and $\l'(r)>0$ for $r\in (0,\bar r)$;
\item \label{cond 2}  $\l'(0)=0, \l''(0)> 0$ and $\l'(r)>0$ for $r\in (0,\bar r)$.
\end{enumerate}
In \ref{cond 1} case, $M$ is complete around $N\times \{0\}$ and this case corresponds to condition (H1'-H2') in \cite{Br}. In \ref{cond 2} case, $M$ has a horizon $N\times \{0\}$ which corresponds to condition (H1-H2) in \cite{Br}.

By a parameter transformation $s=\l(r)$ so that $s\in [\underline s, \bar s)$, the metric is transformed as $$\bar g=\frac{1}{V(s)^2}ds^2+s^2 g_N,$$ where $V(s)=\l'(r)$. It follows that $\l''(r)= \l'(r)\frac{d}{ds}V(s)=\frac{d}{ds}\frac12V(s)^2$, which means $\frac12V(s)^2$, as a function of $s$, is smooth. Thus $\frac{\l'''(r)}{\l'(r)}=\frac{d^2}{ds^2}\frac12V(s)^2$ is a smooth function with respect to $r$.

In this case, $M^n, \bar g$ and $V=\l'(r)$ constitute a typical Riemannian triple.
By a direct computation, see e.g. \cite{Br}, we have 
\begin{eqnarray*}
{\rm Ric}={\rm Ric}_N-(\l'\l''+(n-2)\l^2)g_N-(n-1)\frac{\l''}{\l}dr\otimes dr.
\end{eqnarray*}
and
\begin{eqnarray}\label{hessV}
\frac{\bn^2 V}{V}=\l\l'' g_N+\frac{\l'''}{\l'}dr\otimes dr
\end{eqnarray}
which is smooth on $M$.
Thus
\begin{eqnarray*}
Q&=&\frac{\bd V\bar g-\bn^2 V+V{\rm Ric}}{V}\\&=&({\rm Ric}_N-(n-2)\rho g_N)\\&&+\left(\l^2\frac{\l'''}{\l'}+(n-3)\l\l''+(n-2)\l'(\rho-\l'^2)\right)g_N.
\end{eqnarray*}
for some constant $\rho\in\mathbb{R}$.

As shown in  \cite{Br}, if ${\rm Ric}_N\geq (n-2)\rho g_N$ and the function $2\frac{\l''}{\l}-(n-2)\frac{\rho-\l'^2}{\l^2}$ is non-decreasing for $r\in [0,\bar r)$, then
$VQ=\bd V\bar g-\bn^2 V+V{\rm Ric}\geq 0$, which means $(M^n, \bar g, V)$ is a sub-static Riemannian triple in the language of this paper.

\begin{defi}
    \label{local warped}
A Riemannian triple $(M, \bar g, V)$ is said to have {\em local warped product structure} near the the inner boundary horizons $N_l$  (see Definition \ref{homo})
if locally around $N_l$ the metric has the warped product structure $\bar g=dr^2+\lambda_l^2(r)g_{N_l}, r\in [0,\e_l)$ for small $\e_l>0$, where $\l_l(r)$ is a positive function on $[0,\e_l)$.  
\end{defi}
Assume inside the $\e$-neighborhood of $N_l$, the potential function $V=V_{N_l}(r)=\l'_l(r)$. Then on each horizon, we have $V(0)=\lambda'_l(0)$, $V_{\nu_{N_l}}(0) = -\lambda_l''(0) $, $H_{N_l}(0)=0$, and it follows from \eqref{hessV} that
\begin{align}
    \di\frac{\bd V \bar g-\bn^2V}{V}(\nu_{N_l}, \nu_{N_l})=\di\frac{(n-1)\lambda''(0)}{\lambda(0)}>0,
    \label{RIC}
\end{align}
in regard to \ref{cond 2}.

\subsection{Inner boundary condition} In this paper, the ambient manifold $M^n$ may be complete without boundary, such as $\mathbb R^n$ and $\mathbb H^n$; may be compact with one or more boundaries such as $\mathbb S^n_+$ and the deSitter-Schwarzschild space; may also have a complete end together with one or more compact boundary components, such as the anti-deSitter-Schwarzschild space etc..

We first  modify the topology of $M^n$ by compactifying all of the ends of $M^n$ except the chosen end by adding the points $\{\infty_k\}$. Consequently, any compact orientable hypersurface divides $M^n$ into two regions, an inside (the open set) and an outside which contains the chosen end. It then makes sense to say one closed hypersurface is contained entirely inside another one. We given the following definition.

\begin{defi}\label{def outermost}
    \
        Suppose the boundary of a domain $\O$ has several connected embedded compact hypersurfaces as components, say $\p\O=\di\S\cup \bigcup_{l=1}^\tau N_l$. $\S$ is called the {\em outermost} boundary hypersurface if all the other components $N_l, l=1,\cdots, \tau$, are contained entirely inside it.
\end{defi}

We mainly consider a bounded domain which possesses an outermost boundary component $\S$ together with no or several compact inner boundary components. These inner boundaries are typically horizons, see the definition of {\em horizons} in \cite{B}. For convenience, we also introduce the following {\it inner boundary condition} for a domain $\O$ in $M$.

\begin{defi}\label{homo} 
    Let $(M^n, \bar g, V)$ be a Riemannian triple. Let $\O\subset M^n$ be a domain with smooth boundary $\P\O=\di\S\cup \bigcup_{l=1}^\tau N_l$ where $\S$ is an outermost boundary component and $N_l$ are compact inner boundaries. $\O$ satisfies {\em inner boundary condition} if the following conditions on the inner boundary components hold:

    \begin{enumerate}
        \item [(i)] $N_l$ are minimal hypersurfaces;
        \item [(ii)] The potential function $V>0$ in $\O$ and $V=0$  on each $N_l$, $l=1,\cdots,\tau$;
        \item [(iii)] On each of the inner boundary components (horizons),
\begin{equation}
    \di\frac{\bd V \bar g_{\nu\nu} -\bn_{\nu\nu}  V}{V}(z)> 0
    \label{constant Ricci}
\end{equation}
for any $z\in N_l$.
    \end{enumerate}
\end{defi}

\begin{rema}\label{rema homo}We notice that

    \begin{enumerate}
        \item  $\O$ trivially satisfies the inner boundary condition when $\P\O$ has only one boundary component $\S$. 
        \item For proving theorems in this paper, it is enough to assume the inner boundaries are minimal hypersurfaces instead of totally geodesic.
        \item  If the manifold has a local warped product structure around the horizons with the static potential being $\l'(r)$ where $\l(r)$ is the warping function, then the {\em inner boundary conditions} are automatically satisfied in view of (\ref{RIC}). Moreover, the quotient in (\ref{constant Ricci}) are constants on each horizon.
        \item In case $M^n$ is a static manifold, then the event horizons are all totally geodesic. Moreover by the Gauss equation and static equation,on each inner boundary (event horizon), we have
            \begin{align*}
                \di\frac{\bd V  -\bn_{\nu\nu}  V}{V}=-{\rm Ric}(\nu,\nu)=\frac12({\rm R}_{N_l}-\rm R),
            \end{align*}
            where $\rm R$ is the scalar curvature of the static manifold which is a constant and $ {\rm R}_{N_l}$ is the scalar curvature of the event horizon. This implies that $c_l=-\di\frac{2(n-1)}{n}\frac{1}{\max_{ N_l}{(\rm R}_{N_l}-\rm R)}$ in Theorem \ref{HK theo} since $N_l$ are compact.
    \end{enumerate}
\end{rema}

\subsection{First eigenvalue} It is important for us to prove that the first Dirichlet eigenvalue of the linear operator $\bd-q$ in $\O$ is positive where $q:=\frac{\bd V}{V}$ is a smooth function which is continuous up to  boundary. Recall the first eigenvalue for a bounded domain $\O$ is defined as 
\begin{eqnarray}
  &&  \lambda_1\left(\bd-q,\O\right)= \di \inf_{f\in H^1_o(\O), \|f\|^2_{L^2(\O)}=1}\int_{\O} |\bn f|^2+qf^2. 
    \label{firsteigen def}
\end{eqnarray}

Since $M^n$ may possibly be a manifold with boundary, we denote by $\mathring{M}$ as its interior. Moreover, $M^n$ may possibly contain one or more complete end(s). 
\begin{lemm}
    \label{lemm first} Suppose $V\in C^\infty(\mathring{M})$ and $V>0$ in $\mathring{M}$. Let $q:=\frac{\bd V}{V}$.
    If $\O\subset \mathring{M}$ is a bounded domain, then 
$\l_1\left(\bd-q, \O\right)\ge0.$
        Moreover, if $\O\neq \mathring{M}$, then 
$\l_1\left(\bd-q, \O\right)>0.$
\end{lemm}

\begin{proof}
To prove the first statement, we only need to show that 
\[
\di\int_{\O} |\bn f|^2+\frac{\bd V}{V}f^2\ge 0,
\]
for any $f\in C^\infty_c(\O)$ with  $\|f\|^2_{L^2(\O)}=1$. By integration-by-parts, we obtain
\[
    \di\int_{\O} |\bn f|^2+\frac{\bd V}{V}f^2 = \di\int_{\O}\left|\bn f-\frac{\bn V}{V}f\right|^2\ge0,
\]
where we have used the fact that the closure of the support of $f$ is strictly contained inside $\O$ so that the boundary integration vanishes.

To prove the last statement, we recall the fact that if one bounded domain is strictly contained in another one, e.g., $\O\subsetneq \O'$, then 
\[
\lambda_1(\O)>\lambda_1(\O').
\]
If $\O\neq \mathring{M}$, then there exists a bounded domain $\O'\subset \mathring{M}$ which strictly contains $\O$, i.e., $\O\subsetneq\O'\subset \mathring{M}$, then combining with (i), we have
\[
\lambda_1(\O)>\lambda_1(\O')\ge 0.
\]
\end{proof}

\

\section{An integral formula}\label{sec 3}
In this section we will prove Theorem \ref{theo rf nc}. For that purpose, we first prove a more general integral formula.

\begin{theo}[General Integral Formula]
    \label{theo grf}
    Let $(M, \bar g)$ and $\O$ be as in Theorem \ref{theo rf nc}.
Let $P_{ij}$ be a smooth symmetric $(0,2)$-tensor on $\O$ and $P$ be the trace of $P_{ij}$ with respect to $\bar g$. We define
$A_{ij}(f):=f_{,ij}+\frac{1}{n-1}Pf\bar g_{ij}-fP_{ij}$
    and its trace with respect to $\bar g$ is $A(f)=\bd f+\frac{P}{n-1}f$. 
    
    Then for any smooth functions $V, f\in C^\infty(\overline{\O})$, we have  
\begin{eqnarray} \label{equ10}
        &&\di\int_{\O}{V}[ A(f)^2- |A_{ij}(f)|^2]
        \\&=&\di\int_{\P\O}{V} h(\n z,\n z)+2{V} u\Delta z+{V} Hu^2+{V}_{,\nu}|\n z|^2 \nonumber\\
        &&+\di\int_{\P\O}2{V} zf_{,i}P_{i\nu}-z^2{V}_{,i}P_{i\nu}-z^2{V} P_{\nu i,i} \nonumber\\
&&\di+ \int_{\O} \left[({V}_{,ij}-\bar\Delta {V} \bar g_{ij}-{V} P_{ij})+{V} (R_{ij}-P_{ij})\right]f_{,i}f_{,j} \nonumber\\ &&\di +\int_{\O}\left[P_{ij}({V}_{,ij}+\frac{1}{n-1}P{V}\bar g_{ij}-{V} P_{ij}) +{V} P_{ij,ji}+2{V}_iP_{ij,j}\right]f^2. \nonumber
   \end{eqnarray}
\end{theo}

\begin{rema}
 When $V\equiv 1$ and $P_{ij}=0$, \eqref{equ10} reduces to Reilly's original formula \cite{Re1, Re2}. when $P_{ij}= R^{\kappa}_{ij}=(n-1)\k\bar g_{ij}$, \eqref{equ10} reduces to the integral formula in \cite{QX}.
\end{rema}

\begin{proof}
    It is direct to verify that 
\[
    \begin{array}[]{rll}
        A(f)^2-|A_{ij}(f)|^2 =& ( \bd f)^2 - |f_{,ij}|^2+2ff_{,ij}P_{ij}-f^2P_{ij}^2+\frac{1}{n-1}P^2f^2, \\
    \end{array} 
\]
and thus  
\begin{eqnarray}
  &&  \di\int_{\O}{V} [A(f)^2-|A_{ij}(f)|^2]\\&=& \di\int_{\O}{V} ( \bd f)^2 -{V} |f_{,ij}|^2+2{V} ff_{,ij}P_{ij}+{V} (\frac{1}{n-1}P^2  -P_{ij}^2)f^2.\nonumber
    \label{equ1}
\end{eqnarray}

As in  \cite{QX}, by integration by parts and using the Ricci identity $f_{,ijj}=(\bd f)_{,i}+R_{ij}f_{,j}$, we have
\begin{eqnarray}\label{xeq1}
&&\int_\O V|f_{,ij}|^2=\int_{\p\O}Vf_{,i\nu}f_i +\int_\O  -V_{,j}f_{,ij}f_{,i}- Vf_{,ijj}f_{,i}\\&=&\int_{\p\O} V  f_{,i\nu}f_{,i} dA+\int_\O -V_{,j}(\frac12|\on f|^2)_{,j} -V\left((\ode f)_{,i}- R_{ij}f_{,j}\right)f_{,i}  \nonumber
\\&=&\int_{\p\O} V f_{,i\nu}f_{,i}- \frac12|\on f|^2V_{,\nu} -  V\ode f f_{,\nu} \nonumber\\&&+\int_\O \frac12|\on f|^2\ode V+ V(\ode f)^2 + \ode f  V_{,i} f_{,i}-VR_{ij}f_{,i}f_{,j},\nonumber
\end{eqnarray}
and
\begin{eqnarray}\label{xeq2}
&&\int_\O \ode f  V_{,i} f_{,i}=\int_{\p\O} V_{,i}f_{,i}f_{,\nu} +\int_\O -V_{,ij}f_{,i}f_{,j}-V_{,i} (\frac12|\on f|^2)_{,i}\\
&=&\int_{\p\O} V_{,i}f_{,i}f_{,\nu}-\frac12|\on f|^2V_{,\nu}+\int_\O -V_{,ij}f_{,i}f_{,j}+\frac12|\on f|^2\ode V.\nonumber
\end{eqnarray}
It follows from \eqref{xeq1} and \eqref{xeq2} that
\begin{eqnarray}
  && \di\int_{\O}{V} ( \bd f)^2 -{V} |f_{,ij}|^2\\&=& \di\int_{\p\O}{V}  \bd ff_{,\nu}+{V}_{,\nu} |\bn f|^2-{V} f_{,i}f_{,i\nu}-f_{,\nu}{V}_{,i}f_{,i}\nonumber\\&&+\di\int_\O {V}_{,ij}f_{,i}f_{,j}-\bd {V}|\bn f|^2  +{V} R_{ij}f_{,i}f_{,j}.\nonumber \label{equ2}
\end{eqnarray}

On the other hand, we have
\begin{eqnarray}
        &&\di\int_{\O} 2{V} ff_{,ij}P_{ij}\\&=& \di\int_{\p\O}2{V} ff_{,i}P_{i\nu}+\int_{\O}-{V}_{,i}(f^2)_{,j}P_{ij}-2{V} P_{ij}f_{,i}f_{,j}-2{V} ff_{,i}P_{ij,j}\nonumber\\
        &=& \di\int_{\p\O}2{V} ff_{,i}P_{i\nu}-f^2{V}_{,i}P_{i\nu}\nonumber\\
        &&\di+\int_{\O}f^2{V}_{,ij}P_{ij}+f^2{V}_{,i}P_{ij,j}-2{V} P_{ij}f_{,i}f_{,j}-{V}(f^2)_{,i}P_{ij,j} \nonumber\\
&=& \di\int_{\p\O}2{V} ff_{i}P_{i\nu}-f^2{V}_{i}P_{i\nu}-f^2{V} P_{\nu j,j}\nonumber\\
        &&\di+\int_{\O}f^2{V}_{,ij}P_{ij}-2{V} P_{ij}f_{,i}f_{,j}+2f^2{V}_{,i}P_{ij,j}+{V} f^2 P_{ij,ji}.\nonumber
\label{equ3}
\end{eqnarray}

Inserting  (\ref{equ2}) and (\ref{equ3}) into (\ref{equ1}), we have
\begin{eqnarray}\label{equ4}
        &&\di\int_{\O}{V} [A(f)^2-|A_{ij}(f)|^2]
\\&=& \di\int_{\p\O}{V}  \bd ff_{,\nu}+ |\bn f|^2{V}_{,\nu}-{V} f_{,i\nu}f_{,i}-f_{,\nu}{V}_{,i}f_{,i}\nonumber \\
        &&+\di \int_{\p\O} 2{V} ff_{,i}P_{i\nu}-f^2{V}_{,i}P_{i\nu}-f^2{V} P_{\nu j,j}\nonumber\\
        &&+  \di  \int_{\O}[({V}_{,ij}-\bar\Delta {V} \bar g_{ij}-{V} P_{ij})+{V} (R_{ij}-P_{ij})] f_{,i}f_{,j}\nonumber\\ 
        &&\di+\int_{\O}[P_{ij}({V}_{,ij}+\frac{1}{n-1}P{V}\bar g_{ij}-{V} P_{ij}) +{V} P_{ij,ji}+2{V}_{,i}P_{ij,j}]f^2.\nonumber
        \end{eqnarray}

As in \cite{QX}, we can simplify the boundary terms. 
We choose an orthonormal frame $\{e_i\}_{i=1}^n$ such that $e_n=\nu$ on $M$. Note that $z=f|_{\p\O}$ and $u=f_{,\nu}$.
From the Gauss-Weingarten formula we deduce
\begin{eqnarray}\label{xeq7}
&&\int_{\p\O} V\ode ff_{,\nu}-V f_{,i\nu}f_{,i}=\int_M V f_{,\a\a}f_{,\nu}-V f_{,\a\nu}f_{,\a} \\&=&\int_{\p\O} V\left(u\De z+Hu^2- \n u \n z+h(\n z, \n z)\right).\nonumber
\end{eqnarray}
On the other hand,
\begin{eqnarray}\label{xeq8}
&&\int_{\p\O} |\on f|^2 V_{,\nu}-  f_{,\nu} V_{,i}f_{,i} =\int_{\p\O} |\n z|^2 V_{,\nu}-u\n V \n z\\&=&\int_{\p\O} |\n z|^2 V_\nu+V\n u \n z+Vu\De z.\nonumber
\end{eqnarray}
Combining \eqref{xeq7} and \eqref{xeq8}, we obtain
\begin{eqnarray}\label{equ5}
&&\int_{\p\O}{V}  \bd ff_{,\nu}+ |\bn f|^2{V}_{,\nu}-{V} f_{,i\nu}f_{,i}-f_{,\nu}{V}_{,i}f_{,i}
\\&=& \int_{\p\O}{V} h(\n z,\n z)+2{V} u\Delta z+{V} Hu^2+{V}_{,\nu}|\n z|^2.\nonumber
\end{eqnarray}

Pluging (\ref{equ5}) into (\ref{equ4}), we finish the proof.  
\end{proof} 

\noindent{\it Proof of Theorem \ref{theo rf nc}.} 	Since 
 $\frac{\bn^2 V}{V}$ is smoothly continuous up to $\p\O$, we can choose $$P_{ij}=\frac{1}{V}(V_{,ij}-\bd V \bar g_{ij})$$  in \eqref{equ10} and its trace with respect  $\bar g$ is $P=-\frac{n-1}{V}\bd V$. 
Then $A_{ij}(f)=f_{,ij}-\frac{V_{,ij}}{V}f$ and $A(f)=\bd f-\frac{\bd V}{V}f$.
Recall \begin{equation}
    Q_{ij}=\frac{1}{V}[\bd V \bar g_{ij}-V_{,ij}+VR_{ij}].
    \label{static identity}
\end{equation}
Thus $P_{ij}=R_{ij}-Q_{ij}$.  This implies 
\begin{eqnarray}\label{identity P}
 && V_{,ij}+\frac{1}{n-1}P V \bar g_{ij}-V P_{ij}=0, \quad \bd V+\frac{1}{n-1}PV=0.
\end{eqnarray}

We need the following two identities.
\begin{lemm}
    \label{lemma 1}
    \begin{eqnarray}        \label{ident1}
       &&VP_{ij,ji}+2V_{,i}P_{ij,j} =
        V_{,ji}Q_{ij}+\frac{1}{2}V_{,j}{\rm R}_{,j}.
    \end{eqnarray}
\end{lemm}
\begin{proof}
   By differentiating (\ref{identity P}) twice and using the Ricci identity, we have
 \begin{eqnarray*}
     0&=&\left[V_{,ij}+\frac{1}{n-1}P V \bar g_{ij}-VP_{ij}\right]_{,ji}\\
           &=&\left[(\bd V)_{,i}+R_{ij}V_{,j}+\frac{1}{n-1}(P V)_{,i}-V_{,j}P_{ij}-VP_{ij,j}\right]_{,i}\nonumber\\
            &=&\left[R_{ij}V_{,j}-V_{,j}P_{ij}-VP_{ij,j}\right]_{,i}\nonumber\\
            &=&(R_{ij}-P_{ij})V_{,ji}+R_{ij,i}V_{,j}-2P_{ij,i}V_{,j}-VP_{ij,ji}\nonumber
            \\&=&V_{,ji}Q_{ij}+\frac{1}{2}V_{,j}{\rm R}_{,j}-2P_{ij,i}V_{,j}-VP_{ij,ji}.\nonumber
\end{eqnarray*}
In the last equality we used the contracted second Bianchi identity $R_{ij,i}=\frac12{\rm R}_{,j}$.
\end{proof}

\begin{lemm}
    \label{lemma 2} 
    \begin{equation}
        \di        V_{,j}Q_{ij,i}-\frac{1}{2}V_{,j}{\rm R}_{,j}+\frac{1}{V}Q_{ij}V_{,i}V_{,j}=0.
        \label{equ10.5.5}
    \end{equation}
\end{lemm}
\begin{proof}
  By taking first derivative of \eqref{static identity}, using the Ricci identity and the second Bianchi identity, 
 \begin{eqnarray*}
VQ_{ij,i}+V_{,i}Q_{ij}&=&[\bd V \bar g_{ij}-V_{,ij}+VR_{ij}]_{,i}
\\&=&(\bd V)_{,i}-V_{,iji}+V_{,i}R_{ij}+V R_{ij,i}
\\&=&\frac12V {\rm R}_{,j}.
\end{eqnarray*}
\end{proof}
\begin{rema} It is direct to see from Lemma \ref{lemma 2} that a static manifold with potential $V$ must be of constant scalar curvature at the point with $V=0$. One can show that the set $\{V=0\}$ is indeed a regular embedded hypersurface and thus a static manifold must be of constant scalar curvature everywhere see e.g. \cite{FM, Cor}.
\end{rema}

We continue to prove Theorem \ref{theo rf nc}. We first analyze the second boundary integration in (\ref{equ10}). Using the definition of $P_{ij}$, we have
\begin{eqnarray}\label{equ51}
&&\int_{\p\O}2V ff_{,i}P_{i\nu}-f^2V_{i}P_{i\nu}-f^2V P_{\nu j,j}\\&=&\int_{\p\O}2V ff_{,i}P_{i\nu}-f^2(V P_{i\nu})_{,i}\nonumber
\\&=&\int_{\p\O}2ff_{,i}(V_{,i\nu}-\bd V  \bar g_{i\nu}) -f^2[V_{,i\nu i}-(\bd V)_{,i}]\nonumber
\\&=&\int_{\p\O} 2z\bn^2 V (\n z, \nu)-2zu (\De V+HV_{,\nu})-f^2\bar{R}_{i\nu}V_{,i}.\nonumber
\end{eqnarray}
In the last equality we used the Ricci identity. 

Using \eqref{identity P} and \eqref{ident1}, the last two lines of RHS of \eqref{equ10} reduce to
\begin{eqnarray}\label{equ6}
 &&\int_{\O} VQ_{ij}f_{,i}f_{,j} +(V_{,ij}Q_{ij}+\frac{1}{2}V_{,j}R_{,j})f^2
 \\&=&\int_{\p\O} V_{,i}Q_{i\nu}f^2+
 \int_{\O} VQ_{ij}f_{,i}f_{,j} -V_{,j}Q_{ij,i}f^2-2V_{,j}Q_{ij}ff_{,i}+\frac{1}{2}V_{,j}{\rm R}_{,j}f^2\nonumber
 \\&=&\int_{\p\O} V_{,i}Q_{i\nu}f^2+\int_{\O} VQ_{ij}f_{,i}f_{,j} -2V_{,j}Q_{ij}ff_{,i}+\frac{1}{V}Q_{ij}V_{,i}V_{,j}\nonumber
 \\&=&\int_{\p\O} V_{,i}Q_{i\nu}f^2+\int_{\O} \frac{1}{V}Q_{ij}(Vf_{,i}-V_{,i}f)( Vf_{,j}-V_{,j}f).\nonumber
\end{eqnarray}
In the second equality we also used \eqref{equ10.5.5}.
We arrive at \eqref{equ10.5 nc} by replacing  (\ref{equ51}) and (\ref{equ6}) into (\ref{equ10}). The proof is completed.
\qed

An immediate consequence of Theorem \ref{theo rf nc} applying on a sub-static Riemmannian triple is the following
\begin{coro} Let $(M^n, \bar g, V)$ be an $n$-dimensional  sub-static Riemannian triple and $\O\subset M$ be a bounded domain on which $V>0$, then
\begin{eqnarray}        \label{equ10.5.1}
           && \di\int_{\O}V \left(\bd f-\frac{\bd V}{V}f\right)^2-V\left|\bn^2 f-\frac{\bn^2 V}{V}f\right|^2 d\O\\&\geq &\di\int_{\p\O}V h(\n z,\n z)+2V u\Delta z+V Hu^2+V_{,\nu}|\n z|^2 dA\nonumber\\ 
            &&+\di\int_{\p\O} 2z\bn^2 V(\n z, \nu)-2zu (\De V+HV_{,\nu})dA\nonumber\\&&+\int_{\p\O}z^2(Q-{\rm Ric})(\bn V, \nu)dA.\nonumber
            \end{eqnarray}
where $Q$ is the $(0,2)$-tensor in Definition \ref{defi static}.
\end{coro}

\

\section{Heintze-Karcher type inequality}

In this section, we prove the Heintze-Karcher type inequalities.

\noindent{\it Proof of Theorem \ref{HK theo}.}  
Suppose $\O\neq \mathring{M}$. 
By virtue of Lemma \ref{lemm first}, we know from the standard elliptic PDE theory that the following Dirichlet boundary value problem \begin{equation}\label{Dirichlet3}
\left\{
    \begin{array}[]{rlll}
\di     \bd f-\frac{\bd V}{V}f&=&1&\mathrm{ in }\ \O,\\
        f&=&0 &\mathrm{ on }\ \S,\\
        f&=&\di c_l &\mathrm{ on }\ N_l, \quad l=1,\cdots,\tau,\\
    \end{array}
    \right.
\end{equation}
where $ c_l=\di\max_{z\in N_l}\left\{- \di\frac{n-1}{n}\frac{V}{(\bd V \bar g-\bn^2 V)(\nu_{N_l}, \nu_{N_l})}\right\}<0$   on each $N_l$ $(l=1,\cdots, \tau)$, admits a unique smooth solution $f\in C^{\infty}(\overline{\O})$. Notice that by the inner boundary condition (c.f. Definition \ref{homo}), on each $N_l$, $V=0$, $H_{N_l}=0$ and $c_l<0$. 

It now follows from the integral formula \eqref{equ10.5 nc}, the equation \eqref{Dirichlet3} and the Cauchy-Schwarz inequality that
\begin{eqnarray}\label{case2-1}
&&\frac{n-1}{n}\int_\O V
\\&\geq &\int_{\S} VHu^2+\di \sum ^\tau_{l=1}c_l^2\int_{N_l} V_{,\nu_{N_l}} \frac{(\bd V \bar g-\bar \n^2V)(\nu_{N_l}, \nu_{N_l})}{V} \nonumber
\\
&\ge &\int_{\S} VHu^2-\frac{n-1}{n}\di \sum ^\tau_{l=1}c_l\int_{N_l} V_{,\nu_{N_l}},\nonumber
 \end{eqnarray}
 where we have used  $V_{,\nu_{N_l}}<0$, which is due to the fact that $V=0$ on $N_l$ while $V>0$ inside $\O$. On the other hand, by integration by parts, using the equation \eqref{Dirichlet3}, we have
\begin{eqnarray}\label{case2-2}
    \int_\O V&=&\int_{\S} (Vf_\nu-V_\nu f)+\sum^\tau_{l=1}\int_{N_l} ( Vf_{,\nu_{N_l}}-V_{,\nu_{N_l}} f)
    \\&=&\int_{\S} Vu-\sum^\tau_{l=1}c_l\int_{N_l}  V_{,\nu_{N_l}}. \nonumber
\end{eqnarray}
Combining \eqref{case2-1} and \eqref{case2-2} and using H\"older's inequality, we obtain
\begin{eqnarray}\label{equa1}
&&\left(\int_\O V d\O+\sum^\tau_{l=1}c_l\int_{N_l}  V_{,\nu_{N_l}}\right)^2\\&=&
\left(\int_\S Vu dA\right)^2
\leq \int_\S VHu^2 dA\int_\S \frac{V}{H}dA\nonumber
\\&\leq &\frac{n-1}{n}\left(\int_\O V d\O+\sum^\tau_{l=1}c_l\int_{N_l}  V_{,\nu_{N_l}}\right)\int_\S \frac{V}{H}dA.\nonumber
\end{eqnarray}
When $\left(\int_\O V d\O+\sum^\tau_{l=1}c_l\int_{N_l}  V_{,\nu_{N_l}}\right)$ is non-positive,  \eqref{Heintze-Karcher} is trivial. When $\left(\int_\O V d\O+\sum^\tau_{l=1}c_l\int_{N_l}  V_{,\nu_{N_l}}\right)$ is positive, \eqref{Heintze-Karcher} follows from \eqref{equa1}. 

When $\mathring{M}$ itself is a bounded domain, we can use an approximation method to prove that the Heintze-Karcher inequality still holds for $\O=\mathring{M}$.

We are remained with the equality assertion. Assume the equality in \eqref{Heintze-Karcher} holds, then every inequality appeared in the proof will be identity. In particular, 
    \begin{equation}
f_{,ij}-\frac{V_{,ij}}{V}f =\frac{1}{n}\bar g_{ij},
        \label{equ232 sub}
    \end{equation}
    since we have used the Cauchy-Schwarz inequality. 
Restricting \eqref{equ232 sub} on the outermost boundary $\S$, and using the boundary condition $z=0$, we have
$h_{\a\b}f_{,\nu}=\frac1n g_{\a\b}$, which implies that $\S$ is umbilical.
\qed

\

\noindent{\it Proof of Corollary \ref{coro local}.} We only need to check the following identity holds on the horizons when there exists a local warped product structure:
\begin{eqnarray*}
c_l\int_{N_l} V_{,\nu_{N_l}} dA&=&-\frac{1}{n}\frac{\l_l(0)}{\l_l''(0)}\int_{N_l} -\l_l''(0) \l_l(0)^{n-1}d{\rm Vol}_{g_{N_l}}\\&=&\frac1n\l_l(0)^n{\rm Vol}(N_l, g_{N_l}).
\end{eqnarray*}
\qed

\

\noindent{\it Proof of Theorem \ref{Alex theo}.} Clearly both geodesic slices and umbilical hypersurfaces lying in the region with constant sectional curvature are  of constant mean curvature. 

Conversely, if $\S$ is of constant mean curvature, then the equality holds in the Heintze-Karcher inequality. From Theorem \ref{HK theo},  we know $\S$ is umbilical and thus $h_{\a\b}=cg_{\a\b}$ for some constant $c$. By the same proof as in Theorem 1.1 of \cite{Br}, ${\rm Ric}(e_\a,\nu)=0$ for $e_\a\in T\S$ and thus the normal vector filed $\nu$ is an eigenvector of the Ricci tensor. 

In the rest, we will show $\S$ is either a slice or an umbilical hypersurface lying in the region with constant sectional curvature. By direct computations, we have the following decomposition of Ricci tensor,  see equation (2) in \cite{Br} or equation (1) in \cite{BE},
\begin{align}
    {\rm Ric}=&\di-\Psi(r)\bar g - \Phi(r)dr\otimes dr.
    \label{Ricci tensor}
\end{align}
where $\Psi(r)=\di\frac{\l''}{\l}-(n-2)\frac{\rho-\l'^2}{\l^2}$ and $\Phi(r)=(n-2)\di\left(\frac{\l''}{\l}+\frac{\rho-\l'^2}{\l^2}\right)$.

Choose normal coordinates $\{\P_r, \P_1,\cdots, \P_{n-1}\}$ at a point, where $\P_\a$, $\a=1,\cdots,n-1$ are coordinates of the base manifold $N$. Then
\begin{align}\label{Rrr}
   {\rm Ric}(\P_r) 
= -(\Psi +\Phi)\P_r, \quad 
 {\rm Ric}(\P_\a) 
= -\Psi \P_\a.  
\end{align}
This implies that $\p_r$ is an eigenvector of the Ricci tensor with respect to the eigenvalue $-\Psi$ and $\{\p_1,\cdots, \p_{n-1}\}$ constitutes a basis of the eigenspace of Ricci tensor with respect to the eigenvalue $-(\Psi +\Phi)$. When $\Phi(r)\neq 0$ for some point on $\S$, then Ricci tensor has two distinct eigenvalues. Thus $\nu$, as an eigenvector, is either parallel or orthogonal to $\P_r$ at this point. 

In general, we denote $S:=\{x\in \S: \Phi(r(x))\neq 0\}$. There are two alternatives: (i) $S\neq \emptyset$; (ii) $S=\emptyset$.

If $S\neq \emptyset$, there exists some point $p$ and its neighborhood $\mathcal{N}_p$ in $S$. We claim $\p_r$ is parallel to $\nu$ on $S$. In fact, if otherwise $\p_r\perp \nu$, then $\p_r$ is a tangent  vector field for points in $S$ and $c=h(\p_r,\p_r)=\<\bn_{\p_r}\p_r,\nu\>=0$, which is impossible. Therefore $\p_r$ is parallel to $\nu$ on $S$ and   $\Phi(r(x))$ is a constant on $S$.  We have now $S$ is a non-empty open and closed set.
It follows from the connectedness of $\S$ that $S=\S$. Hence $\p_r$ is parallel to $\nu$ on $\S$ which implies $\S$ is a geodesic slice.

If  $S=\emptyset$, then $\S$ is contained in the region where $\Phi(r)=0$. 
By computation, we have
\begin{align}\label{Riemancurvature}
R_{ijij}=\l (r)^2(\rho-\l'(r)^2),\quad R_{irir}=-\l(r)\l''(r),
\end{align}
where $\{\p_i\}$ are coordinates vector for $N$ and $R_{ijij}$ is the Riemannian curvature tensor. It is direct to see from \eqref{Riemancurvature} that the points with $\Phi(r)=0$ are isotropic. We conclude from Schur's Theorem that $\S$ is contained in the region with constant sectional curvature.
The proof is completed.
\qed

\section{Minkowski type inequality}
In this section, we prove the Minkowski type inequality. We first make some rearrangement for the boundary terms in the integral formula \eqref{equ10.5 nc}.

\begin{coro}\label{coro rf}
   Let $M$, $\O$, $V$ and $f$ as in Theorem \ref{theo rf nc}. If $\O$ has  a connected boundary component $\S$ on which $V>0$, then
    \begin{eqnarray}\label{equ10.5 Minkowski}
             && \di\int_{\O}V \left(\bd f-\frac{\bd V}{V}f\right)^2-V\left|\bn^2 f-\frac{\bn^2 V}{V}f\right|^2 d\O\\ 
            &=&\int_{\S}V \left[h-\frac{V_{,\nu}}{V} g\right]\left(\n z-\frac{\n V}{V}z, \n z-\frac{\n V}{V}z\right) dA\nonumber\\ 
            &&+\di\int_{\S}VH\left(u-\frac{V_{,\nu}}{V}z\right)^2+2V\left(u-\frac{V_{,\nu}}{V}z\right)\left(\De z-\frac{\De V}{V}z\right) dA\nonumber\\
&&+\di\int_{\p\O\setminus\S}V h(\n z,\n z)+2V u\Delta z+V Hu^2+V_{,\nu}|\n z|^2 dA\nonumber\\ 
&&+\di\int_{\p\O\setminus\S} 2z\bn^2 V(\n z, \nu)-2zu (\De V+HV_{,\nu}) dA \nonumber\\
            &&+\di\int_{\p\O\setminus\S}z^2\frac{\bd V \bar g-\bn^2 V}{V}(\bn V, \nu) dA\nonumber\\
            &&\di+  \int_{\O}\left(\bd V \bar g-\bn^2 V+V {\rm Ric}\right)\left(\bn f-\frac{\bn V}{V}f, \bn f-\frac{\bn V}{V}f\right) d\O.\nonumber
\end{eqnarray}
\end{coro}

\begin{proof}
We only need to show the integral over $\S$ in \eqref{equ10.5 nc} has the form stated in \eqref{equ10.5 Minkowski}.

It follows from the Gauss-Weingarten formula that on $\S$,
\begin{eqnarray}\label{static1}
&&V_{,\a\nu}=\n_\a V_{,\nu}-h_{\a\b}V_\b,
\quad\quad \bd V-V_{,\nu\nu}=\De V+HV_{,\nu}.
\end{eqnarray}
Inserting \eqref{static1} into the part of the integral on $\S$ in \eqref{equ10.5 nc}, we deduce
\begin{eqnarray}\label{equ10.5.2}
             &&\di\int_{\S}V h(\n z,\n z)+2V u\Delta z+V Hu^2+V_{,\nu}|\n z|^2 \\ 
&&+\di\int_{\S} 2z\bn^2 V(\n z, \nu)-2zu (\De V+HV_{,\nu})\nonumber\\
 &&+\di\int_{\S}z^2\frac{\bd V \bar g-\bn^2 V}{V}(\bn V, \nu) dA\nonumber\\ &=&\di\int_{\S}V h(\n z,\n z)+2V u\Delta z+V Hu^2+V_{,\nu}|\n z|^2 \nonumber\\ 
            &&+\di\int_{\S}\left(2zu-z^2\frac{V_{,\nu}}{V}\right)(-\De V-HV_{,\nu})\nonumber\\&&+\di\int_{\S}\left(2zz_\a-z^2\frac{V_{\a}}{V}\right)(\n_\a V_{,\nu}-h_{\a\b}V_\b).\nonumber
\end{eqnarray}
By intergration by parts,
\begin{eqnarray} \label{equ10.5.3}
            && \di\int_{\S}\left(2zz_\a-z^2\frac{V_{\a}}{V}\right)\n_\a V_{,\nu}\\&=&     \di\int_{\S}\left(-2|\n z|^2-2z\De z+\frac{2zz_\a V_{\a}}{V}+z^2\frac{\De V}{V}- z^2\frac{|\n V|^2}{V^2}\right) V_{,\nu}. \nonumber 
\end{eqnarray}

By some tedious calculation of completing the square for the terms in \eqref{equ10.5.2} and \eqref{equ10.5.3}, we get our conclusion. The proof is completed.
\end{proof}

\noindent\textit{Proof of Theorem \ref{rm22}.} 
Consider the Neumann boundary value problem
\begin{eqnarray}\label{Neumann}
\left\{
    \begin{array}[]{rlll}
\di      \bd f-\frac{\bd V}{V}f&=&1&\mathrm{in}\ \O,\\
        Vf_{,\nu}-V_{,\nu}f&=&cV &\mathrm{ on }\ \S,
    \end{array}
    \right.
\end{eqnarray}
where $c=\frac{\int_\O V}{\int_\S V}$.
The existence and uniqueness (up to an additive $\a V$) of the solution to \eqref{Neumann} is due to the Fredholm alternative. In fact, it is easy to see that the Neumann problem \eqref{Neumann} is equivalent to 
\begin{eqnarray*}
\left\{
    \begin{array}[]{rlll}
\di      \div_{\bar g}(V^2\bn w)&=&V&\mathrm{in}\ \O,\\
        V^2 \bn_{\nu} w&=&c V &\mathrm{ on }\ \S,
    \end{array}
    \right.
\end{eqnarray*}
by the correspondence $f=wV$.

We will apply the solution $f$ of $\eqref{Neumann}$ to the integral formula \eqref{equ10.5 Minkowski}.

By using the Cauchy-Schwarz inequality, the equation in \eqref{Neumann} and the sub-static condition \eqref{static potential inequality}, we have from \eqref{equ10.5 Minkowski} that
\begin{eqnarray}\label{R1}
&&\frac{n-1}{n}\int_\O V\\&\geq & \int_{{\S}}V \left[h-\frac{V_{,\nu}}{V} g\right]\left(\n z-\frac{\n V}{V}z, \n z-\frac{\n V}{V}z\right)\nonumber\\ 
            &&+\di\int_{{\S}}VH\left(u-\frac{V_{,\nu}}{V}z\right)^2+2V\left(u-\frac{V_{,\nu}}{V}z\right)\left(\De z-\frac{\De V}{V}z\right). \nonumber\end{eqnarray}
By the assumption \eqref{ass},
the first term in RHS of \eqref{R1} is nonnegative.
Therefore, we derive from \eqref{R1} and the boundary condition that
\begin{eqnarray*}
&&\frac{n-1}{n}\int_\O V\nonumber\ge\int_{\S} c^2HV+\int_\S 2c(V\De z-z\De V)=\frac{\left(\int_\O V\right)^2}{\left(\int_\S V\right)^2} \int_{\S} HV.\nonumber
\end{eqnarray*}
It follows that
\begin{eqnarray*}
\left(\int_\S V dA \right)^2\geq \frac{n}{n-1}\int_\O V d\O\int_{\S} HV dA.
\end{eqnarray*}


Assume now  the equality in \eqref{rm2} holds. Then each step above is an equality and we have 
\begin{eqnarray}\label{equality111}
f_{,ij}-\frac{V_{,ij}}{V}f=\frac1n \bar g_{ij}\hbox{ in }\O
\end{eqnarray}
and \begin{eqnarray}\label{equality222}
(h_{\a\b}-\frac{V_{,\nu}}{V}g_{\a\b})\left(\frac{z}{V}\right)_{\a}\left(\frac{z}{V}\right)_{\b}\geq 0\hbox{ on }\S.
\end{eqnarray}

Let $\mathcal{S}:=\{x\in\S: (h_{\a\b}-\frac{V_{,\nu}}{V}g_{\a\b})(x)>0\}$. By the assumption, $\mathcal{S}$ is non-empty, say $p\in \mathcal{S}$. It is clear that $\mathcal{S}$ is open. Thus there is a neighborhood $\mathcal{N}_p\subset \S$ such that $(h_{\a\b}-\frac{V_{,\nu}}{V}g_{\a\b})>0$ in $\mathcal{N}_p$. It follows from \eqref{equality222} that $z=\a V$ for some $\a\in \mathbb{R}$ in $\mathcal{N}_p\subset \S$. We also have from \eqref{equality111} that $\tilde{f}:=f-\a V$ satisfies
\begin{eqnarray*}
\tilde{f}_{,ij}-\frac{V_{,ij}}{V}\tilde{f}=\frac1n \bar g_{ij}.
\end{eqnarray*}
Restrcting on $\mathcal{N}_p$ and using $\tilde{f}=0$ and $\tilde{f}_{,\nu}=c$, we see $h_{\a\b}=\frac{1}{nc}g_{\a\b}$ in $\mathcal{N}_p$. This implies $\mathcal{S}$ is also closed. Hence $\mathcal{S}=\S$. By the same argument applied on $\S$, we find $\S$ is umbilical and of constant mean curvature. 
\qed 


\

\noindent\textit{Proof of Theorem \ref{rm33}.} 
Consider the following Neumann boundary value problem
\begin{eqnarray}\label{Neumann2.1}
\left\{
    \begin{array}[]{rlll}
\di       \bd f-\frac{\bd V}{V}f&=&1&\mathrm{in}\ \O,\\
        Vf_{,\nu}-fV_{,\nu}&=&c_1 V &\mathrm{ on }\ \S,\\
        Vf_{,\nu_{\p M}}-fV_{,\nu_{\p M}}&=&c_0 V_{,\nu_{\p M}} &\mathrm{ on }\ \p M,
    \end{array}
    \right.
\end{eqnarray}
where $c_1=\frac{\int_\O Vd\O+\frac1n\l(0)^n{\rm Vol}(N, g_N)}{\int_\S V}$ and $c_0=-\frac1n \frac{\l(0)}{\l''(0)}$.
Similar argument as in the proof of Theorem \ref{rm22} yields the existence of of a unique solution $f\in C^{\infty}(\overline{\O})$ (up to an additive $\a V$) to \eqref{Neumann2.1}.


We  apply the solution $f$ of $\eqref{Neumann2.1}$ to the integral formula \eqref{equ10.5 Minkowski}.

 We take care of the integral over $\p\O\setminus\S=\p M$ in \eqref{equ10.5 Minkowski}.
Since $V=0$  and $V_{,\nu_{\p M}}=-\l''(0)\neq 0$ on $\p M$, we see from the boundary condition on $\p M$ that  $z=c_0$ on $\p M$. Note also $H=0$ on $\p M$.
It follows that  the integral over $\p\O\setminus\S=\p M$ in \eqref{equ10.5 Minkowski}
\begin{eqnarray*}
&&=\int_{\p M} c_0^2\frac{\bd V\bar g-\bn^2 V}{V}(V_{,\p_r}\p_r, -\p_r)=-\frac{n-1}{n^2}\l(0)^n{\rm Vol}(N, g_N).
\end{eqnarray*}
Here we also used \eqref{hessV}.

Using the sub-static condition \eqref{static potential inequality} and the assumption \eqref{ass} on $\S$, we derive from \eqref{equ10.5 Minkowski} that
\begin{eqnarray}\label{R1.1}
&&\frac{n-1}{n}\int_\O V\\&\geq &
            \di\int_{{\S}}VH\left(u-\frac{V_{,\nu}}{V}z\right)^2+2V\left(u-\frac{V_{,\nu}}{V}z\right)\left(\De z-\frac{\De V}{V}z\right)\nonumber
            \\&&-\frac{n-1}{n^2}\l(0)^n{\rm Vol}(N, g_N). \nonumber\end{eqnarray}
Using the boundary condition on $\S$, we deduce from \eqref{R1.1} that
\begin{eqnarray}
&&\frac{n-1}{n}\left(\int_\O V d\O+\frac1n\l(0)^n{\rm Vol}(N, g_N)\right)\nonumber\\&\geq &\int_{\S} c_1^2HV+\int_\S 2c_1(V\De z-z\De V)\nonumber\\&=&\frac{\left(\int_\O Vd\O+\frac1n\l(0)^n{\rm Vol}(N, g_N)\right)^2}{\left(\int_\S V\right)^2} \int_{\S} HV.\nonumber
\end{eqnarray}
It follows that
\begin{eqnarray*}
\left(\int_\S V dA \right)^2\geq \frac{n}{n-1}\left(\int_\O V d\O+\frac1n\l(0)^n{\rm Vol}(N, g_N)\right)\int_{\S} HV dA.
\end{eqnarray*}

We are remained with the equality assertion.
We first show a geodesic slice $N\times \{r\}$ satisfies the equality in \eqref{case2Minkowski}. Let $X=\l(r)\p_r$ be a conformal Killing vector field such that $\bn X=V\bar g$ and in turn $\n X= V\bar g-\<X, \nu\>h$. Thus we have the Minkowski identity
\begin{eqnarray}\label{Minkowski identity}
&&\int_\S H\<X, \nu\>= (n-1)\int_\S V.
\end{eqnarray}
$H$ is a constant on a slice and hence
\begin{eqnarray*}
(n-1)\int_\S V&=&H\int_\S \<X, \nu\>=
H\left(n\int_\O Vd\O+ \l(0)^n{\rm Vol}(N, g_N)\right).
\end{eqnarray*}
It follows then the equality in \eqref{case2Minkowski} holds.

Conversely, if the equality holds, the same argument as in the proof of Theorem \ref{rm22} shows that $\S$ is of constant mean curvature.
It follows from Theorem \ref{Alex theo}  that $\S$ is a slice. The proof is completed.
\qed

\

\section{Alexandrov-Fenchel type inequality in space forms}
In this section we prove the weighted Alexandrov-Fenchel type inequality.

\noindent{\it Proof of Theorem \ref{almost Schur}.}
Let $K=1$ in the case $\S\subset \SS^n_+$ and $K=-1$ in the case $\S\subset \HH^n$.
We  claim that under our assumption, $\S$ is a closed sub-static manifold, namely \begin{eqnarray}\label{hypersub-static}
\De V g_{\a\b}-V_{\a\b}+V R^{\S}_{\a\b}\geq 0,
\end{eqnarray}
where $R^{\S}_{\a\b}$ is the Ricci tensor of $\S$.

Indeed, by the Gauss formula and the Gauss equation, we have
$$\De V=\bd V- V_{,\nu\nu}-HV_{,\nu} = -K(n-1)V-HV_{,\nu},$$ $$V_{\a\b}=V_{,\a\b}-V_{,\nu}h_{\a\b}=-KV g_{\a\b}-V_{,\nu}h_{\a\b},$$
$$ R^{\S}_{\a\b}=Hh_{\a\b}-h_{\a\g}h_{\b\g}+K(n-2)g_{\a\b},$$
Thus
\begin{eqnarray}\label{convexity}
&&\De V g_{\a\b}-V_{\a\b}+V R^{\S}_{\a\b}=(Vh_{\b\g}-V_{,\nu}g_{\b\g})(Hg_{\a\g}-h_{\a\g}).
\end{eqnarray}
In the case $\S\subset \SS^n_+$, $\S$ is convex. Thus  $Hg_{\a\g}-h_{\a\g}\geq 0$ and $h_{\b\g}\geq 0>\frac{V_{,\nu}}{V}g_{\b\g}$.
In the case $\S\subset \SS^n_+$,  $h_{\b\g}-\frac{V_{,\nu}}{V}g_{\b\g}\geq 0$ by assumption and $h_{\b\g}>0$ follows from this.
In view of \eqref{convexity}, we have \eqref{hypersub-static} in both cases. 

Let $f\in C^\infty (\S)$ be the unique solution of 
\begin{eqnarray}\label{close eq}
\De f-\frac{\De V}{V}f= H-\overline{H}^V \hbox{ on }\S, \quad \int_\S f dA=0.
\end{eqnarray}
The existence follows again from the Fredholm alternative.

For notation simplicity, we denote by $A_{\a\b}=\n^2_{\a\b} f-\frac{\n^2_{\a\b}V}{V} f$, $A=\De f-\frac{\De V}{V}f$ and $\mathring{A}_{\a\b}=A_{\a\b}-\frac{1}{n-1}A g_{\a\b}$.
By integrating by parts, we have
\begin{eqnarray}\label{close eq1}
&&\int_\S V\left|H-\overline{H}^V\right|^2
\\&=&\int_\S (H-\overline{H}^V)(V\De f-f\De V)\nonumber
\\&=&\int_\S -\n H(V\n f-f\n V)\nonumber
\\&=&\int_\S -\frac{n-1}{n-2}\n_\a\left(h_{\a\b}-\frac{H}{n-1}g_{\a\b}\right)(V\n_\b f-f\n_\b V)\nonumber
\\&=& \frac{n-1}{n-2}\int_\S V\left(h_{\a\b}-\frac{H}{n-1}g_{\a\b}\right)\mathring{A}_{\a\b}\nonumber
\\&\leq & \frac{n-1}{n-2}\left(\int_\S V\left|h_{\a\b}-\frac{H}{n-1}g_{\a\b}\right|^2 \right)^{\frac12}  \left(\int_\S V|\mathring{A}_{\a\b}|^2\right)^{\frac12}.\nonumber
\end{eqnarray}
In the third equality, we used the Codazzi property of $h_{\a\b}$, which is a well-known fact for hypersurfaces in space forms. In the last inequality we used the H\"older inequality.

On the other hand, applying the integral formula \eqref{equ10.5 nc} for $\O=\S$ (note that $\S$ is an $(n-1)$-dimensional closed manifold)  and using \eqref{hypersub-static}, we have
\begin{eqnarray*}
\int_\S V \left(\De f-\frac{\De V}{V}f\right)^2 \geq \int_\S V\left| f_{\a\b}-\frac{V_{\a\b}}{V}f\right|^2.
\end{eqnarray*}
It follows that \begin{eqnarray}\label{reilly no bdry}
\int_\S V|\mathring{A}_{\a\b}|^2\leq \frac{n-2}{n-1}\int_\S V \left(\De f-\frac{\De V}{V}f\right)^2.
\end{eqnarray}

Combining \eqref{close eq1} and \eqref{reilly no bdry}, using again the equation \eqref{close eq}, we conclude \eqref{AF1} holds.
Inequality \eqref{AF2} follows from \eqref{AF1} by an algebra computation. This is a key observation due to Ge-Wang \cite{GW}.

Next we prove the assertion for the equality for the case $\HH^n$.  Obviously, if $\S$ is a geodesic ball, then $H$ and $\sigma_2(h)$ are both constants and the equality holds.

Conversely, if the equality in \eqref{AF1} holds, by checking the above proof, we see that
\begin{eqnarray}\label{equal1}
h_{\a\b}-\frac{H}{n-1}g_{\a\b}=\mathring{A}_{\a\b},
\end{eqnarray}
and 
\begin{eqnarray}\label{equal2}
\left(\De V  g-\n^2 V+V {\rm Ric}^\S\right)\left(\n \frac{f}{V}, \n \frac{f}{V}\right)=0.
\end{eqnarray}
Since $\S$ is closed, we know there exists at least one elliptic point $p$ at which all the principal curvatures are larger than $1$. It follows from \eqref{convexity} that $\De V  g-\n^2 V+V {\rm Ric}^\S$ is positive-definite at $p$. From continuity we know that $\De V  g-\n^2 V+V {\rm Ric}^\S$ is  positive-definite in a neighborhood $\mathcal{N}_p$ of $p$. In view of \eqref{equal2}, we have $f=V$ in $\mathcal{N}_p$. It follows then from \eqref{equal1} that $h_{\a\b}=\frac{H}{n-1}g_{\a\b}=c g_{\a\b}$ in $\mathcal{N}_p$ for some constant $c>1$. Using \eqref{convexity} again, we see that
$\De V g-\n^2 V+V {\rm Ric}^{\S} =\frac{n-2}{n-1}c(cV-V_{,\nu})g$ is positive-definite in $\overline{\mathcal{N}_p}$. Therefore the set $\mathcal{S}:=\{x\in\S:  \De V g-\n^2 V+V {\rm Ric}^{\S}>0\}$ is open and closed and thus $\mathcal{S}=\S$ due to the connectedness of $\S$. Applying the above argument to all point in $\S$, we conclude that $\S$ is umbilical.

The proof for the case $\SS_+^n$ is almost the same. Thus we omit it. The proof is completed.
\qed

\

A similar consideration as Theorem \ref{almost Schur} gives a weighted version of  de Lellis-Topping type almost Schur lemma. \begin{theo}
Let $(M^n, \bar g, V)$ be an $n$-dimensional ($n\geq 3$) closed sub-static Riemannian triple. Then the following inequality holds:
\begin{eqnarray}\label{almost Schur2}
\int_M \left|{\rm R}-\overline{{\rm R}}^V\right|^2 \leq \frac{4n(n-1)}{(n-2)^2}\int_M V \left|{\rm Ric}-\frac{{\rm R}}{n}\bar g\right|^2,
\end{eqnarray}
where $\overline{{\rm R}}^V=\frac{\int_M V{\rm R}}{\int_M V}$. 
Moreover, if there is one point in $M$ at which the strict sub-static inequality holds, then equality in \eqref{almost Schur2} holds if and only if $M$ is Einstein.
\end{theo}

Since the proof is the same as  that of Theorem  \ref{almost Schur}. So we omit it.

\

\

{\bf Acknowledgements.} This paper has been done while both authors visited Department of Mathematics at McGill University. The results have been presented at McGill Geometric Analysis Seminar. We would like to thank the department for its hospitality, Professor Pengfei Guan for his constant support, and Professor Niky Kamran for his interests. 

\

\end{document}